\documentclass[11pt, letter]{amsart}

\usepackage{fullpage}			
\usepackage[english]{babel}
\usepackage[utf8]{inputenc}		
\usepackage[T1]{fontenc}		
\usepackage{amsmath}			
\usepackage{amssymb}			
\usepackage{amsthm}			
\usepackage{mathrsfs}			
\usepackage[pdftex,bookmarksnumbered]{hyperref}	
\usepackage{eucal}
\usepackage{amsfonts}
\usepackage{multicol}

\usepackage{tikz-cd}			
\usepackage[numbers]{natbib}

\newcounter{generalnumbering} \numberwithin{generalnumbering}{section}

\theoremstyle{plain}		\newtheorem{theorem}[generalnumbering]{Theorem}
\theoremstyle{plain}		\newtheorem{corollary}[generalnumbering]{Corollary}
\theoremstyle{definition}		\newtheorem{definition}[generalnumbering]{Definition}
\theoremstyle{definition}		\newtheorem{example}[generalnumbering]{Example}
\theoremstyle{plain}		\newtheorem{proposition}[generalnumbering]{Proposition}
\theoremstyle{plain}		\newtheorem{lemma}[generalnumbering]{Lemma}

\newenvironment{remark}
{\vspace{\topsep}\noindent\textbf{Remark.}}
{\vspace{\topsep}}

\author{Luiz Cordeiro}

\address{Department of Mathematics and Statistics,
University of Ottawa,
585 King Edward Ave.,
Ottawa, ON K1N 6N5,
Canada}
\email{lcord081@uottawa.ca}

\title{An elementary approach to sofic groupoids}


\title{An elementary approach to sofic groupoids}

\begin{document}

\begin{abstract}
We describe sofic groupoids in elementary terms and prove several permanence properties for soficity. We show that soficity can be determined in terms of the full group alone, answering a question by Conley, Kechris and Tucker-Drob.

\end{abstract}

\maketitle

\section*{Introduction}

Sofic groups were first considered by Gromov \cite{MR1694588} in his work on Symbolic Dynamics (originally under the nable ``\emph{initially subamenable groups}``), and in 2010, Elek and Lippner \cite{eleklippner2010} introduced soficity for equivalence relations. Since then, the classes of sofic groups and equivalence relations have been shown to satisfy several important conjectures,see for example \cite{MR3408561,eleklippner2010,MR2089244,elekszabo2005}.

The original definitions of soficity are graph-theoretical, but alternative definitions by Ozawa \cite{ozawasoficnotes} and P\v{a}unescu \cite{paunescu2011} describe soficity at the level of the so-called full semigroup of $R$, or in terms of the action of the full group on the measure algebra, which can be immediately generalized to groupoids. We will describe general elementary techniques to deal with (abstract) sofic groupoids.


\section{Groupoids}

A \emph{groupoid} is a small category with inverses. More precisely, it consists of a set $G$ together with a partially defined binary operation $G^{(2)}\to G$, $(g,h)\mapsto gh$, where $G^{(2)}\subseteq G\times G$, called \emph{product}, satisfying
\begin{enumerate}
\item[(1)] If $(g,h),(h,k)\in G^{(2)}$ then $(gh,k),(g,hk)\in G^{(2)}$ and $g(hk)=(gh)k$;
\item[(2)] For all $g\in G$, there exists $g'\in G$ such that $(g,g'),(g',g)\in G^{(2)}$, and if $(g,h),(k,g)\in G^{(2)}$ then $g'(gh)=h$ and $(kg)g'=k$.
\end{enumerate}
Given $g_1,g_2,\ldots,g_n$ such that $(g_i,g_{i+1})\in G^{(2)}$, the product $g_1\cdots g_n$ is uniquely determined by (1), and the element $g'$ in (2) is unique -- we denote it $g^{-1}$ and call it the \emph{inverse} of $g$. The \emph{source} and \emph{range} of $g\in G$ are $s(g)=g^{-1}g$ and $r(g)=gg^{-1}$, respectively. The \emph{unit space} of $G$ is $G^{(0)}=s(G)=r(G)$. We then obtain $G^{(2)}=\left\{(g,h)\in G\times G:s(g)=r(h)\right\}$.


A \emph{discrete measurable groupoid} is a groupoid $G$ with a standard Borel structure for which the product and inverse maps are Borel and $s^{-1}(x)$ is countable for all $x\in G^{(0)}$. In this case the source and range maps are also Borel, and $G^{(2)}$ and $G^{(0)}=\left\{x\in G:x=s(x)\right\}$ are Borel subsets of $G$.

The \emph{Borel full semigroup} of a discrete measurable groupoid $G$ is the set $[[G]]_B$ of Borel subsets $\alpha\subseteq G$ such that the restrictions $s|_\alpha$ and $r|_\alpha$ of the source and range maps are injections, and hence Borel isomorphisms onto their respective images \cite[Theorem 4.12.4]{srivastavaacourseonborelsets}. Moreover, a simple application of the Lusin-Novikov Theorem \cite[Theorem 5.10.3]{srivastavaacourseonborelsets} implies that $G$ can be covered by countably many elements of $[[G]]_B$.

$[[G]]_B$ is an inverse monoid\footnote{An \emph{inverse monoid} is a set $M$ with an associative binary operation $(x,y)\mapsto xy$, which has a neutral element $1$ and such that for each element $g\in M$ there is an unique element $h\in M$ satisfying $g=ghg$ and $h=hgh$, called the \emph{inverse} of $g$ and denoted $h=g^{-1}$.} with the natural product and inverse of sets, namely
\[\alpha\beta=\left\{ab:(a,b)\in (\alpha\times\beta)\cap G^{(2)}\right\},\qquad \alpha^{-1}=\left\{a^{-1}:a\in\alpha\right\}\]
and $G^{(0)}$ is the unit of $[[G]]_B$, which we may instead denote by $G^{(0)}=1$ or $1_G$. Moreover, $[[G]]_B$ is closed below, i.e., if $\beta\subseteq\alpha$ and $\alpha\in[[G]]_B$ then $\beta\in[[G]]_B$.

A \emph{probability measure-preserving} (\emph{pmp}) groupoid is a discrete measurable grou\-poid $G$ with a Borel measure $\mu$ on $G^{(0)}$ satisfying $\mu(s(\alpha))=\mu(r(\alpha))$ for all $\alpha\in[[G]]_B$. 
We write $(G,\mu)$ for a pmp groupoid when we need the measure $\mu$ to be explicit.

The measure $\mu$ induces a pseudometric $d_{\mu}$ on $[[G]]_B$ via
\[d_{\mu}(\alpha,\beta)=\mu(s(\alpha\triangle\beta))=\mu(r(\alpha\triangle\beta)).\]

The \emph{trace} of $\alpha\in[[G]]_B$ is defined as $\operatorname{tr}(\alpha)=\mu(\alpha\cap G^{(0)})$. For us, it will be easier to deal with the trace instead of the metric above, which is allowed by Proposition \ref{theoremtraceanddistance} below.

The following properties of $d_\mu$ are useful, and we leave the proof to the interested reader.
\begin{proposition}\label{propositionmetric}
Given $\alpha,\beta,\gamma,\delta\in [[G]]_B$;
\begin{enumerate}
\item $d_{\mu}(\alpha,\beta)=d_{\mu}(\alpha^{-1},\beta^{-1})$.
\item $d_{\mu}(\alpha\beta,\gamma\delta)\leq d_{\mu}(\alpha,\gamma)+d_\mu(\beta,\delta)$;
\item $d_{\mu}(\alpha,\beta^{-1})\leq d_{\mu}(\alpha,\alpha\beta\alpha)+d_{\mu}(\beta,\beta\alpha\beta)$;
\end{enumerate}
\end{proposition}

The (measured) \emph{full semigroup} of a pmp groupoid $(G,\mu)$ is the metric quotient $[[G]]$ (or $[[G]]_\mu$ to make $\mu$ explicit) of $[[G]]_B$ under the pseudometric $d_\mu$. The proposition above implies that $[[G]]$ is also an inverse monoid with the natural structure.

The \emph{Borel full group} $[G]_B$ of a discrete measurable groupoid $G$ is the set of those $\alpha\in[[G]]_B$ with $s(\alpha)=r(\alpha)=G^{(0)}$. If $(G,\mu)$ is a pmp groupoid, the image of $[G]_B$ in $[[G]]$ is called the (measured) \emph{full group} $G$ and is denoted $[G]$ or $[G]_\mu$.

We will not make a distinction between measured and Borel full semigroups and groups, except when necessary.

\begin{example}
Let $\Gamma$ be a countable group acting on a probability space $(X,\mu)$ by measure-preserving automorphisms. The \emph{transformation groupoid} $G=\Gamma\ltimes X$ is defined as $\Gamma\times X$ with product
\[(h,gx)(g,x)=(hg,x),\quad x\in X,\quad g,h\in\Gamma\qquad\qquad
\begin{tikzcd}[row sep=tiny]
x\arrow[r,"g"]\arrow[rr,bend right,"hg",swap]&gx\arrow[r,"h"]&hgx
\end{tikzcd}\]
In this case, the unit space $G^{(0)}$ can be identified with $X$.

\textbf{Subexample 1:} If $X=\left\{*\right\}$ is a singleton we retrieve the group $\Gamma$, in which case $[G]=\Gamma$ and $[[G]]$ consists of $\Gamma$ and a zero (absorbing) element.

\textbf{Subexample 2:} If $\Gamma=1$ is a trivial group we retrieve $X$, $[[G]]$ is the measure algebra of $X$ and $[G]$ is the trivial group.

\end{example}

\begin{example}\label{exampleequivalencerelations}
A measure-preserving equivalence relation $R$ on a probability space $(X,\mu)$ can be regarded as a groupoid with operation $(x,y)(y,z)=(x,z)$ whenever $(x,y),(y,z)\in R$. We can identify the unit space $R^{(0)}$ as $X$ and $[[R]]$ as the set of all partial Borel automorphisms $f:\operatorname{dom}(f)\to\operatorname{ran}(f)$, where $\operatorname{dom}(f),\operatorname{ran}(f)\subseteq X$, such that $\operatorname{graph}(f)\subseteq R$. Namely, to each such $f$ we associate the element $\left\{(f(x),x):x\in\operatorname{dom}(f)\right\}$ of $[[R]]$. The product becomes the usual composition of partial maps, and the trace becomes $\operatorname{tr}(f)=\mu\left\{x:f(x)=x\right\}$.
\end{example}

If $\left\{(G_n,\mu_n)\right\}_n$ is a family of pmp groupoids and $t_n$ are non-negative numbers such that $\sum_n t_n=1$, we construct the \emph{convex combination groupoid} $G=\sum t_n G_n$ as follows: $G$ is the disjoint union of all $G_n$, $G^{(2)}$ is the disjoint union of $G_n^{(2)}$, the product on $G$ restricts to the product on each $G_n$, and the measure $\mu$ on $G$ is given by $\mu(A)=\sum_n t_n\mu_n(A\cap G_n)$.

\subsubsection*{Finite groupoids}

Every finite groupoid is a convex combination of groupoids of the form $G=\Gamma\times Y^2$, where $\Gamma$ is a (finite) group, $Y$ is a (finite) set and $Y^2$ is the largest equivalence relation on $Y$. These are the \emph{connected} finite groupoids. We see both $\Gamma$ and $Y^2$ as groupoids on their own right, and the product has the obvious groupoid structure. In this case, $G^{(0)}=Y$, and the only probability measure on $Y$ which makes $G$ pmp is the normalized counting measure, $\mu_\#(A)=|A|/|Y|$.

We will analyse the full semigroup $[[Y^2]]$ as in Example \ref{exampleequivalencerelations}.

\begin{proposition}\label{propositionfinitegroupoids}
\begin{enumerate}
\item[(a)] If $G$ is a connected finite pmp groupoid, then there exists a finite set $Y$ and an isometric embedding $\pi:[[G]]\to[[Y^2]]$.
\item[(b)] If $G$ is a finite pmp groupoid and $\epsilon>0$, then there exists a finite set $Y$ and a map $\pi:[[G]]\to[[Y^2]]$ such that $d(\pi(\alpha\beta),\pi(\alpha)\pi(\beta))<\epsilon$ and $|\operatorname{tr}(\alpha)-\operatorname{tr}(\pi(\alpha))|<\epsilon$ for all $\alpha\in[[G]]$.
\end{enumerate}
\end{proposition}
\begin{proof}
\begin{enumerate}
\item[(a)] Suppose $G=\Gamma\times Y^2$. For every $(g,(y,x))\in G$, set $\pi(g,(y,x))\in[[(H\times Y)^2]]$ by $\operatorname{dom}(\pi(g,(y,x)))=H\times\{x\}$ and $\pi(g,(y,x))(h,x)=(gh,y)$. Then $\pi:[[G]]\to[[(H\times Y)^2]]$, $\pi(\alpha)=\bigcup_{g\in\alpha}\pi(g)$ has the desired properties.
\item[(b)] Suppose $G=tH+(1-t)K$, where $H,K$ are connected finite pmp groupoids, and suppose further that $t$ is rational, say $t=p/q$, $p,q\in\mathbb{N}$. Take finite sets $X,Y$ and isometric embeddings $\pi_H:[[H]]\to[[X^2]]$ and $\pi_K:[[K]]\to[[Y^2]]$. Let $[q]=\left\{0,1,\ldots,q-1\right\}$ be a finite set with $q$ elements, and set $\pi:[[G]]\to[[\left([q]\times X\times Y\right)^2]]$ by
\[\pi(\alpha)(j,x,y)=\begin{cases}(j,\pi_H(\alpha\cap H)x,y)&\text{ if }j\leq p-1\\
(j,x,\pi_K(\alpha\cap K)y)&\text{ if }p\leq j\leq q-1\end{cases}\]
(and the domain of $\pi(\alpha)$ consists of all $(j,x,y)$ for which we can apply the definition above.) Then $\pi$ is an isometric embedding.

We can iterate the argument above to finite rational convex combinations. In the non-rational case, we approximate the coefficients in the convex combination by rational numbers and obtain approximate embeddings as in the proposition.\qedhere
\end{enumerate}
\end{proof}

\subsubsection*{Ultraproducts}
The language of metric ultraproducts is useful for soficity, and we'll describe them briefly here. We refer to \cite{pk12} and \cite{MR3408561} for the details. Let $(M_k,d_k)$ be a sequence of metric spaces of diameter $\leq 1$, and $\mathcal{U}$ a free ultrafilter on $\mathbb{N}$. The \emph{metric ultraproduct} of $(M_n,d_n)$ along $\mathcal{U}$ is the metric quotient of $\prod_k M_k$ under the pseudometric $d_\mathcal{U}((x_k),(y_k))=\lim_{k\to\mathcal{U}}d_k(x_k,y_k)$, and we denote it $\prod_{\mathcal{U}}M_k$. We denote the class of a sequence $(x_k)_k\in\prod_k M_k$ by $(x_k)_\mathcal{U}$.


If $(G_k,\mu_k)$ is a family of pmp groupoids, the trace on $\prod_{\mathcal{U}}[[G_k]]$ is given by
\[\operatorname{tr}(g_k)_\mathcal{U}=\lim_{k\to\mathcal{U}}\operatorname{tr}(g_k).\]
Moreover, by \ref{propositionmetric} $\prod_{\mathcal{U}}[[G_k]]$ is an inverse monoid with respect to the canonical product, namely $(g_k)_\mathcal{U}(h_k)_{\mathcal{U}}=(g_kh_k)_{\mathcal{U}}$.

\begin{remark}
One can avoid ultraproducts when dealing with sofic groupoids as follows: For every $n\in\mathbb{N}$, let $Y_n=\left\{0,\ldots,n-1\right\}$ be a finite set with $n$ elements, and denote $[[n]]=[[Y_n^2]]$. Consider the product space $\prod[[n]]$ endowed with the supremum metric and define an equivalence relation $\sim$ on $\prod[[n]]$ by setting
\[(x_n)\sim (y_n)\qquad\text{iff}\qquad\lim_{n\to\infty}d_\#(x_n,y_n)=0.\]
Denote by $\prod^{\ell^\infty/c_0}[[n]]=\prod[[n]]/\sim$ the quotient. Proposition \ref{propositionmetric} also implies that $\prod^{\ell^\infty/c_0}[[n]]$ is an inverse monoid with the obvious operations.

We can naturally embed $[[n]]$ into $[[n+1]]$ (because $Y_n\subseteq Y_{n+1}$), and this modifies the metric by at most $\frac{1}{n+1}$. Also, $[[n]]$ embeds isometrically into $[[kn]]$ as follows: Given $\alpha\in[[n]]$, set $\pi(\alpha)\in[[kn]]$ as $\pi(\alpha)(qn+j)=qn+\alpha(j)$, whenever $0\leq q\leq k-1$ and $j\in\operatorname{dom}(\alpha)$.

This way, we can embed $[[n]]$ into any $p\geq n$ as follows: if $p=qn+r$, with $0\leq r<n$, embed $[[n]]$ into $[[qn]]$ and then into $[[qn+1]],[[qn+2]],\ldots,[[qn+r]$. The metric changes by at most $\frac{1}{qn+r}+\cdots+\frac{1}{qn+1}\leq\frac{n}{qn}=\frac{n}{p-r}\leq\frac{n}{p-n}$, and this goes to $0$ as $p\to\infty$. With these embeddings and a couple of diagonal arguments, one easily proves the following:
\end{remark}

\begin{theorem}\label{theoremlinfinityc0}
A separable metric space (semigroup) $M$ embeds isometrically into $\prod_{\mathcal{U}} [[n_k]]$ (where $n_k$ is a sequence of natural numbers and $\mathcal{U}$ is an ultrafilter on $\mathbb{N}$) if and only if $M$ embeds into $\prod^{\ell^\infty/c_0}[[n]]$.
\end{theorem}

In particular, the choice of free ultrafilter $\mathcal{U}$ or sequence $(n_k)$ does not matter for the existence of an embedding into $\prod_{\mathcal{U}}[[n_k]]$.

We will be dealing with ultraproducts of full semigroups, and most natural operations extend to ultraproducts. For example, if $\alpha,\beta\in[[G]]$ are such that $\beta^{-1}\alpha$ and $\beta\alpha^{-1}$ are idempotents, then $\alpha\cup\beta\in[[G]]$. We can consider similar unions in ultraproducts under the same hypotheses. We won't make further reference to these facts during the remainder of the paper.

\section{Sofic groupoids}

We fix, for the remainder of this paper, a free ultrafilter $\mathcal{U}$ on $\mathbb{N}$.

\begin{definition}
A pmp groupoid $G$ is \emph{sofic} if there exists a sequence $\left\{G_k\right\}_k$ of finite pmp groupoids and an isometric embedding $\pi:[[G]]\to\prod_{\mathcal{U}}[[G_k]]$.
\end{definition}

Equivalently, a pmp groupoid $G$ is sofic if and only if for every $\epsilon>0$ and every finite subset $K$ of $[[G]]$ (or $[[G]]_B$, for that matter), there exist a finite groupoid $H$ and a map $\pi:[[G]]\to[[H]]$ such that $d(\pi(\alpha\beta),\pi(\alpha)\pi(\beta))<\epsilon$ and $|\operatorname{tr}(\alpha)-\operatorname{tr}(\pi(\alpha))|<\epsilon$ for all $\alpha,\beta\in K$. The map $\pi$ is called a $(K,\epsilon)$-almost morphism.

This definition differs from the usual one (see \cite{ozawasoficnotes} or \cite{MR3035288}) on the initial choice of finite models, but Proposition \ref{propositionfinitegroupoids} and Theorem \ref{theoremlinfinityc0} show that they are equivalent.

\begin{proposition}\label{theoremtraceanddistance}
An embedding $\Phi:M\to\prod_{\mathcal{U}}[[G_k]]$ from any sub-inverse monoid $M$ of $[[G]]$ is isometric if and only if it preserves the trace.
\end{proposition}
\begin{proof}
We simply need to write the distance in terms of the trace and vice versa. First one verifies that if $\Phi$ is isometric then $\Phi(1)=1$, since this is the only element of trace $1$, and then that
\[\operatorname{tr}(\alpha)=1-d_\mu(s(\alpha),1)-d_\mu(s(\alpha),\alpha).\]
For the converse, one uses
\begin{align*}
d_\mu(\alpha,\beta)
&=\operatorname{tr}(s(\alpha))+\operatorname{tr}(s(\beta))-\operatorname{tr}(s(\alpha)s(\beta))-\operatorname{tr}(\beta^{-1}\alpha).\qedhere
\end{align*}
\end{proof}

\begin{remark}\label{increasingunionofsoficequivalencerelations}
If $\left\{G_n\right\}_n$ is an increasing sequence of sofic groupoids, then $G=\bigcup_{n=1}^\infty G_n$ is also sofic. Indeed, $\left\{[[G_n]]\right\}_n$ is an increasing sequence of semigroups of $[[G]]$ with dense union, so almost morphisms of each $[[G_n]]$ give us the necessary almost morphisms of $[[G]]$.
\end{remark}

\section{Permanence properties}

In this section we will be concerned with permanence properties of the class of sofic groupoids. We will simply say that a measure $\mu$ on a discrete measurable groupoid $G$ is sofic if $(G,\mu)$ is a (pmp) sofic groupoid.

Given a non-null subgroupoid $H$ of $G$, denote by $\mu_H$ the normalized measure on $H^{(0)}$, i.e., $\mu_H(A)=\mu(A)/\mu(H^{(0)})$ for $A\subseteq H^{(0)}$, and by $\operatorname{tr}_H$ for the corresponding trace on $[[H]]$.

\begin{theorem}\label{theorempermanence}
Let $G$ be a discrete measurable groupoid.
\begin{enumerate}
\item If $\mu$ is a strong limit of sofic measures\footnote{%
Recall that a net $\left\{\mu_i\right\}_{i\in I}$ of measures on a measurable space $(X,\mathcal{B})$ converges strongly to a measure $\mu$ if $\mu_i(A)\to \mu(A)$ for all $A\in\mathcal{B}$.%
}, then $\mu$ is sofic as well..
\item A countable convex combination of sofic measures is sofic.
\item If $\mu$ has a disintegration of the form $\mu=\int_{G^{(0)}} p_xd\nu(x)$, where $\nu$-a.e. $p_x$ is a probability measure such that $(G,p_x)$ is sofic, then $(G,\mu)$ is also sofic. In particular, if a.e.\ ergodic component of $(G,\mu)$ is sofic, so is $(G,\mu)$.
\item\label{theorempermanenceitemsubgroupoid} If $(G,\mu)$ is sofic and $H$ is a non-null subgroupoid of $G$ then $(H,\mu_H)$ is sofic.
\item If $\nu\ll\mu$, where $(G,\nu)$ is pmp, and $\mu$ is sofic, then $\nu$ is sofic.
\item If $\left\{H_n\right\}$ is a countable Borel partition of $G$ by non-null subgroupoids, then $G$ is sofic if and only if each $H_n$ is sofic.

\end{enumerate}
\end{theorem}
\begin{proof}
Item 1.\ is clear since soficity is an approximation property for the measure.
\begin{enumerate}
\item[2.] Suppose $\nu,\rho$ are sofic measures and $\mu=t\nu+(1-t)\rho$, $0<t<1$. Take sofic embeddings $\Phi_\nu:[[G]]_\nu\to\prod_\mathcal{U}[[G_k]]$ and $\Phi_\rho:[[G]]_\rho\to\prod_\mathcal{U}[[H_k]]$. Set $\Phi:[[G]]_\mu\to\prod_\mathcal{U}[[tG_k+(1-t)H_k]]$ as
\[\Phi(\alpha)=(\Phi_\nu(\alpha))\cup (\Phi_\rho(\alpha))\]
which, it is easy to check, is a sofic embedding. The countable infinite case follows from 1.
\item[3.] From the previous items it suffices to check that $\mu$ is a limit of convex combinations of sofic $p_x$. Let $K$ be a finite collection of Borel subsets of $G^{(0)}$ and $\epsilon>0$. The maps $x\mapsto\operatorname{tr}_x(A)$, $A\in K$, take values in $[0,1]$, so by partitioning $[0,1]$ and taking preimages, we can find a finite partition $\left\{X_j\right\}_{j=1}^N$ of $G^{(0)}$ for which $|p_x(A)-p_y(A)|<\epsilon$ for all $A\in K$ whenever $x$ and $y$ belong to the same $X_j$. For each non-null $X_j$, choose $x(j)\in X_j$ with $p_{x(j)}$ sofic. Then for $A\in K$,
\begin{align*}
\mu(A)&=\int_{G^{(0)}}p_x(A)d\nu(x)=\sum_j\left(\int_{X_j}p_{x(j)}(A)d\nu(x)\right)\pm\epsilon\\
&=\left(\sum_j\nu(X_j)p_{x(j)}\right)(A)\pm\epsilon.
\end{align*}
\item[4.] Let $K\subseteq[[H]]$ be a finite subset and $\epsilon>0$. Since $[[H]]$ is contained in $[[G]]$ (as a semigroup, but with a different metric), there exists a $(K,\epsilon)$-almost morphism $\theta:[[G]]\to[[F]]$ for some finite pmp groupoid $F$. We may assume that $1_H\in K$, and that $\theta(1_H)$ is an idempotent in $[[F]]$. For $\alpha\in[[H]]$, we have $H^{(0)}\alpha H^{(0)}=\alpha$, so substituting $\theta(\alpha)$ by $\theta(1_H)\theta(\alpha)\theta(1_H)$ (and making $\epsilon$ smaller) if necessary, we can further assume that $\theta(\alpha)$ is contained in $F':=\theta(1_H)F\theta(1_H)$, which is a subgroupoid of $F$. This defines a map $\theta_H:[[H]]\to[[F']]$.

To see that $\theta_H$ approximately preserves the trace, note that the trace on $[[H]]$ and the trace on $[[F']]$ are given respectively by
\[\operatorname{tr}_{H}(\alpha)=\frac{\operatorname{tr}_{\mu}(\alpha)}{\operatorname{tr}_{\mu}(1_H)}\qquad\text{and}\qquad\operatorname{tr}_{F'}(\theta_H(\alpha))=\frac{\operatorname{tr}_{F}(\theta(\alpha))}{\operatorname{tr}_{F}(\theta(1_H))},\]
and these numbers are as close as necessary if $\epsilon$ is small enough. Products are dealt with similarly, so $\theta_H$ is an approximate morphism as necessary.
\item[5.] Let $\epsilon>0$. Let $f=d\nu/d\mu$. Take a countable partition $X_1,X_2,\ldots$ of $G^{(0)}$ such that $|f(x)-f(y)|<\epsilon$ whenever $x$ and $y$ belong to the same $X_j$, and fix points $x(j)\in X_j$. Then for all $A\subseteq X$,
\[\nu(A)=\sum_j\int_{X_j\cap A}f(x(j))d\mu(x)\pm\epsilon=\sum_jf(x(j))\mu(X_j)\mu_j(A)\pm\epsilon,\]
where $\mu_j$ is the normalized measure on $X_j$. Since $1=\nu(X)=\sum_j f(x_j)\mu(X_j)\pm\epsilon$, it is not hard to obtain
\[\nu(A)=\sum_j\left(\frac{f(x_j)\mu(X_j)}{\sum_i f(x_i)\mu(X_i)}\right)\mu_j(A)\pm2\frac{\epsilon}{1\pm\epsilon}\]
Each $\mu_j$ is sofic by item 4., so items 1.\ and 2.\ imply that $\nu$ is sofic.
\item[6.] Apply items 4.\ and 2.\ with the fact that $G=\sum_j \mu(H_j)H_j$.\qedhere
\end{enumerate}
\end{proof}

Now we will deal with finite-index subgroupoids.
\begin{definition}
A subgroupoid $H\subseteq G$ is said to have \emph{finite index} in $G$ if there exist $\psi_1,\ldots,\psi_n\in[G]$ such that $\left\{\psi_iH:i=1\ldots n\right\}$ is a partition of $G$. We will call $\psi_1,\ldots,\psi_n$ \emph{left transversals} of $H$ in $G$.
\end{definition}
This definition restricts to the usual notions of finite index subgroups and equivalence relations, as defined in \cite{MR1007409}, in the ergodic case. Note that if $H\subseteq G$ is of finite index then $H^{(0)}=G^{(0)}$.

\begin{theorem}\label{theoremfiniteindexgroupoids}
Suppose $H\subseteq G$ is of finite index. If $H$ is sofic, so is $G$.
\end{theorem}
\begin{proof}
Suppose that $\psi_1,\ldots,\psi_N$ are left transversals for $H\subseteq G$. For each $\alpha\in [[G]]$, let $\alpha_{i,j}=\psi_i^{-1}\alpha\psi_j\cap H$. Note that $\alpha_{i,j}\in[[H]]$ and that $\alpha_{i,j}^{-1}=\alpha_{j,i}$. Moreover, $\alpha_{i,j}\alpha_{k,l}=\varnothing$ whenever $j\neq l$. Let $Y$ be a set with $N$ elements. Given $(i,j)\in Y^2$, let $E_{i,j}=\left\{(i,j)\right\}\in[[Y^2]]$ (as a partial transformation on $Y$, $E_{i,j}$ is simply defined by $E_{i,j}(j)=i$.

Given $k$ and $\gamma\in[[G_k]]$, set $\gamma\otimes E_{i,j}=\gamma\times E_{i,j}\in[[G_k\times Y^2]]$, and then define $\Xi:[[G]]\to\prod_{\mathcal{U}}[[G_k\times Y^2]]$ by
\[\Xi(\alpha)=\bigcup_{i,j}\Phi(\alpha_{i,j})\otimes E_{i,j}\]

First let's show that $\Xi$ is well-defined, or more precisely that the terms in the right-hand side have disjoint sources and ranges: Let $(i,j)$ and $(k,l)$ be given. Then
\[\left(\Phi(\alpha_{i,j})\otimes E_{i,j}\right)\left(\Phi(\alpha_{k,l})\otimes E_{k,l}\right)^{-1}=\Phi(\alpha_{i,j}\alpha_{l,k})\otimes(E_{i,j}E_{l,k})\]
If $j\neq l$ the right-hand side is empty, so assume $j=l$. Then
\[\alpha_{i,j}\alpha_{j,k}=(\psi_i^{-1}\alpha\psi_j\cap H)(\psi_j^{-1}\alpha\psi_k\cap H)\tag{1}\]
If this product is nonempty, then we have $p_i\in\psi_i$, $p_j,q_j\in\psi_j$, $q_k\in\psi_k$ and $g,h\in \alpha$ such that the product $(p_i^{-1}gp_j)(q_j^{-1}hq_k)$ is defined, and both terms belong to $H$. But in particular $s(p_j)=r(q_j^{-1})=s(q_j)$, so $p_j=q_j$. Similarly $g=h$, and so $p_i^{-1}q_k\in H$, thus $q_k\in\psi_i H\cap\psi_k H$, which implies $i=k$.

This proves that the ranges of the terms in the definition of $\Xi(\alpha)$ are disjoint. The sources are dealt with similarly, and so $\Xi$ is well-defined.

Now we need to show that $\Xi$ is a morphism. Suppose $\alpha,\beta\in[[G]]$. We have
\[\Xi(\alpha)\Xi(\beta)=\bigcup_{i,j,k,l}\Phi(\alpha_{i,j}\beta_{k,l})\otimes (E_{i,j}E_{k,l})=\bigcup_{i,j,l}\Phi(\alpha_{i,j}\beta_{j,l})\otimes E_{i,l}.\]
On the other hand $\Xi(\alpha\beta)=\bigcup_{i,l}\Phi((\alpha\beta)_{i,l})\otimes E_{i,l}$, so we are done if we show that for given $i,l$,
\[\bigcup_j\alpha_{i,j}\beta_{j,l}=(\alpha\beta)_{i,l}.\]

The inclusion $\subseteq$ is quite straightforward, using a similar argument to the one right after (1) above. For the converse, suppose $p_i^{-1}abp_l\in(\alpha\beta)_{i,l}$, where $p_i\in\psi_i$, $p_l\in\psi_l$, $a\in\alpha$ and $b\in\beta$. Choose $j$ such that $bp_l\in\psi_j H$, so there is a unique $p_j\in\psi_j$ such that the product $p_j^{-1}bp_l$ is defined and in $H$. Therefore 
\[p_i^{-1}abp_l=(p_i^{-1}a p_j)(p_j^{-1}bp_l)\in\alpha_{i,j}\beta_{j,l}.\]

Finally, we need to show that $\Xi$ is trace-preserving. Note that
\[\operatorname{tr}\Xi(\alpha)=\frac{1}{N}\sum_{i=1}^N\operatorname{tr}(\alpha_{i,i}),\]
so we are done if we prove that $\operatorname{tr}(\alpha_{i,i})=\operatorname{tr}(\alpha)$.

Let's show that $\alpha_{i,i}\cap G^{(0)}=s\left\{g\in\psi_i:r(g)\in\alpha\cap G^{(0)}\right\}$. An element of $\alpha_{i,i}\cap G^{(0)}$ has the form $x=p_i^{-1}ap_i$ for $a\in\alpha$ and $p_i\in\psi_i$. It follows that $x=s(p_i)$, so $r(p_i)=a\in\alpha\cap G^{(0)}$. Conversely, if $x=s(g)$, where $g\in\psi_i$ and $r(g)\in\alpha\cap G^{(0)}$, then $x=g^{-1}r(g)g\in\psi_i^{-1}\alpha\psi_i\cap G^{(0)}$.

Finally, we obtain
\begin{align*}
\operatorname{tr}(\alpha_{i,i})&=\mu(\psi_i^{-1}\alpha\psi_i\cap G^{(0)})=\mu(s(\psi_i\cap r^{-1}(\alpha\cap G^{(0)})))=\mu(r(\psi_i\cap r^{-1}(\alpha\cap G^{(0)})))\\
&=\mu(\alpha\cap G^{(0)})=\operatorname{tr}(\alpha),
\end{align*}
because $r|_{\psi_i}:\psi_i\to G^{(0)}$ is surjective.\qedhere
\end{proof}

We will say that a pmp groupoid $G$ is \emph{periodic} if $s^{-1}(x)$ is finite for all $x\in G^{(0)}$, and that $G$ is \emph{hyperfinite} if it is an increasing union of subgroupoids with finite fibers (this is the measured analogue of the AF groupoids introduced in \cite{MR584266})

\begin{corollary}\label{periodicequivalencerelationsaresofic}
Every hyperfinite groupoid is sofic.
\end{corollary}
\begin{proof}
First note that every measure space $(X,\mu)$, seen as a trivial groupoid (i.e., $X=X^{(0)}$), is sofic. Indeed, $\mu=\int_X\delta_x\mu(x)$, where $\delta_x$ is the point-mass measure on $x$, and $(X,\delta_x)$, as a pmp groupoid, is isomorphic to a singleton, hence finite and sofic.

Suppose $G$ has finite fibers. Let $G_n=\left\{g\in G:|s^{-1}(s(g))|=n\right\}$. Then the $G_n$ are subgroupoids of $G$ with $G=\bigcup_nG_n$, so it suffices to show that each $G_n$ is sofic. An application of Lusin-Novikov implies that the subgroupoid $G_n^{(0)}$, which is simply a measure space, has finite index in $G_n$, which is therefore finite. The general case follows from the remark after Proposition \ref{theoremtraceanddistance}.\qedhere
\end{proof}

\begin{theorem}
Two pmp groupoids $(G,\mu)$ and $(H,\nu)$ are sofic if and only if $(G\times H,\mu\times\nu)$ is sofic.
\end{theorem}
\begin{proof}
One direction is clear, since $[[G]]$ embeds isometrically into $[[G\times H]]$ via $\alpha\mapsto\alpha\times H^{(0)}$, and similarly for $[[H]]$.

Let $M$ be the submonoid of $[[G\times H]]$ of elements of the form $\bigcup_{i=1}^n\alpha_i\times\beta_i$, where $\alpha_i\in[[G]]$, $\beta_i\in[[H]]$, and for $i\neq j$, $s(\alpha_i)\cap s(\alpha_j)=\varnothing$ or $s(\beta_i)\cap s(\beta_j)=\varnothing$, and $r(\alpha_i)\cap r(\alpha_j)=\varnothing$ or $r(\beta_i)\cap r(\beta_j)=\varnothing$. Let's show that $M$ is dense in $[[G\times H]]$.

Let $\phi\in[[G\times H]]$ and $\epsilon>0$. We can take $\alpha_i\in[[G]]$ and $\beta_i\in[[H]]$ such that $(\mu\otimes\nu)(\phi\triangle(\bigcup\alpha_i\times\beta_i))<\epsilon$, and with the $\alpha_i\times\beta_i$ disjoint. For $i\neq j$, let 
\[h_{i,j}=s|_{\alpha_i\times\beta_i}^{-1}(s(\alpha_j\times\beta_j))=\alpha_is(\alpha_j)\times\beta_is(\beta_j)\qquad\text{and}\qquad h_i=\bigcup_{j\neq i}h_{i,j}.\]

Let $x\in h_{i,j}\cap\phi$, so $s(x)=s(a_j,b_j)$ for some $(a_j,b_j)\in\alpha_j\times\beta_j$. If $(a_j,b_j)\in\phi$, we'd obtain $(a_j,b_j)=x\in\alpha_i\times\beta_i$, which happens only if $i=j$. This proves that
\[s\left(\bigcup_i\left(h_i\cap\phi\right)\right)\subseteq s\left(\bigcup_j(\alpha_j\times\beta_j)\setminus\phi\right)\]
In particular, $d(\phi,\phi\setminus\bigcup_ih_i)<\epsilon$, from which follows that $d(\bigcup_i(\alpha_i\times\beta_i\setminus h_i),\phi)<2\epsilon$.

Since each $h_i$ is a rectangle, we can rewrite $\bigcup_i(\alpha_i\times\beta_i\setminus h_i)$ as a union of disjoint rectangles with the desired property for the source map. To deal with the range map one can apply the same argument to $\phi^{-1}$ and take intersections. 

So given sofic embeddings $\Phi:[[G]]\to\prod_\mathcal{U}[[G_n]]$ and $\Psi:[[H]]\to\prod_\mathcal{U}[[H_n]]$ set $\Phi\otimes\Psi:M\to[[G_n\times H_n]]$ by
\[\Phi\otimes\Psi(\bigcup\alpha_i\times\beta_i)=\bigcup\Phi(\alpha_i)\times\Psi(\beta_i),\]
where the $\alpha_i$ and $\beta_i$ satisfy the condition in the definition of $M$.

The element in the right-hand side is well-defined and doesn't depend on the choice of $\alpha_i$ and $\beta_i$ since sofic embeddings preserve sources, ranges, and intersections. $\Phi\otimes\Psi$ is then an trace-preserving embedding of $M$, and hence extends uniquely to a isometric embeddings of $[[G\times H]]$.\qedhere
\end{proof}

\section{Soficity and the full group}

Let's fix some notation here as well. Given a probability space $X$, we denote its measure algebra (i.e., the algebra of measurable subsets of $X$ modulo null sets) by $\operatorname{MAlg}(X)$. Given a pmp groupoid $(G,\mu)$, the set of idempotents of $[[G]]$ (i.e., elements $\alpha$ such that $\alpha^2=\alpha)$ is precisely $\operatorname{MAlg}(G^{(0)})$.

Given $\alpha\in[[G]]$, define $\operatorname{fix}\alpha=\alpha\cap G^{(0)}$ and $\operatorname{supp}\alpha=s(\alpha)\setminus\operatorname{fix}\alpha$. These defines maps $\operatorname{fix},\operatorname{supp}:[[G]]\to\operatorname{MAlg}(G^{(0)})$, and we can extend these maps to ultraproducts of these semigroups.

Again, for each $n\in\mathbb{N}$, fix $Y_n$ a set with $n$ elements, consider the full equivalence relation $Y_n^2$, whose full group $[Y_n^2]$ can be identified as the permutation group $\mathfrak{S}_n$ on $n$ elements, as in Example \ref{exampleequivalencerelations}. We denote the measure algebra of $Y_n$ by $\operatorname{MAlg}(n)$.

A well-known theorem of Dye \cite{MR0158048} states that when $R$ is an aperiodic equivalence relation, the full group $[R]$ completely determines $R$. With this in mind, we prove that a pmp groupoid $G$ is sofic if and only if $[G]$ embeds isometrically into $\prod_{\mathcal{U}}\mathfrak{S}_n$, as long as $G$ doesn't contain ``trivial parts''. This solves a question posed by Conley, Kechris and Tucker-Drob in \cite{MR3035288} in this case.

\begin{definition}
A metric group $(\Gamma,d)$ is \emph{metrically sofic} if it embeds isometrically into $\prod_\mathcal{U}\mathfrak{S}_n$.
\end{definition}

We will need a few technical lemmas relating the full group $[G]$, the measure algebra $\operatorname{MAlg}(G^{(0)})$ and the full semigroup $[[G]]$.

\begin{lemma}\label{lemmasupports}
Let $\theta:[G]\to\prod_{\mathcal{U}}\mathfrak{S}_n$ be an isometric embedding and $\alpha,\beta\in[G]$.
\begin{enumerate}
\item\label{lemmasupportandfix} $\operatorname{supp}\alpha=\operatorname{fix}\beta$ if and only if $\operatorname{supp}(\theta(\alpha))=\operatorname{fix}(\theta(\beta))$.
\item\label{lemmadisjointsupports} $\operatorname{supp}\alpha\cap\operatorname{supp}\beta=\varnothing$ if and only if $\operatorname{supp}\theta(\alpha) \cap\operatorname{supp}\theta(\beta)=\varnothing$
\end{enumerate}
\end{lemma}
\begin{proof}
\begin{enumerate}
\item Simply note that $\operatorname{supp}\alpha=\operatorname{fix}\beta$ is equivalent to
\[d_\mu(\alpha,\beta)=1\qquad\text{and}\qquad\operatorname{tr}(\alpha)+\operatorname{tr}(\beta)=1,\]
and the same condition applies to ultraproducts.
\item We have $\operatorname{supp}\alpha\cap\operatorname{supp}\beta=\varnothing$ if and only if $d_\mu(\alpha,\beta)=d_\mu(1,\alpha)+d_\mu(1,\beta)$, and this condition is preserved by $\theta$.\qedhere
\end{enumerate}
\end{proof}

We will say that a groupoid $G$ is \emph{aperiodic} if $|s^{-1}(x)|=\infty$ for all $x\in G^{(0)}$.

\begin{lemma}\label{lemmaaperiodicsupport}
Suppose $(G,\mu)$ is an aperiodic groupoid. Then for all $A\in\operatorname{MAlg}(G^{(0)})$, there exists $\alpha\in[G]$ such that $\operatorname{supp}\alpha=A$.
\end{lemma}
\begin{proof}
We can decompose $G$ as $G=H+K$, where $H$ and $K$ are subgroupoids, $|r(s^{-1}(x))|=\infty$ for all $x\in H^{(0)}$ and for every $x\in K^{(0)}$, $K^x_x=s^{-1}(x)\cap r^{-1}(x)$ is infinite.

The equivalence relation $R(H)=(r,s)(H)$ on $H^{(0)}$ is aperiodic, in the usual sense, so if $A\subseteq H^{(0)}$, \cite[Lemma 4.10]{MR2583950} allows us to take $f\in[[R(H)]]$ with $\operatorname{supp}(f)=A$, and $f$ lifts to an element of $[[H]]$.

If $A\subseteq K^{(0)}$, we use the fact that $K$ is covered by countably many elements of $[[K]]$, and it is easy to construct $\alpha\in[K]$ with $s(g)=r(g)$ for all $g\in\alpha$ and $\alpha\cap K^{(0)}=K^{(0)}\setminus A$.\qedhere
\end{proof}

\begin{lemma}\label{lemmaextensionfullsemigrouptofullgroup}
If $\gamma\in[[G]]$, then there exists $\widetilde{\gamma}\in[G]$ with $\gamma\subseteq\widetilde{\gamma}$.
\end{lemma}
\begin{proof}[Sketch of proof]
Let $\gamma_0=\gamma$, and for all $n\geq 1$, set
\[\gamma_n=\left\{g\in\gamma^{-n}:s(g)\not\in s(\gamma), r(g)\not\in r(\gamma)\right\}\]
One shows that $s(\gamma_n)\cap s(\gamma_m)=r(\gamma_n)\cap r(\gamma_m)=\varnothing$ for $n\neq m$, and that $\bigcup_n s(\gamma_n)$ and $\bigcup_n r(\gamma_n)$ are both contained and conull in $s(\gamma)\cup r(\gamma)$. Therefore $\widetilde{\gamma}=\bigcup_n\gamma_n\cup(G^{(0)}\setminus(s(\gamma)\cup r(\gamma))$ has the desired properties.\qedhere
\end{proof}

If $\Gamma_1$ and $\Gamma_2$ are groups acting on sets $X_1$ and $X_2$, respectively, $\theta:\Gamma_1\to\Gamma_2$ is a homomorphism and $\phi:X_1\to X_2$ is a function, we say that the pair $(\theta,\phi)$ is \emph{covariant} if $\phi(\gamma x)=\theta(\gamma)\phi(x)$ for all $\gamma\in\Gamma_1$ and $x\in X_1$.

Given $\alpha\in[G]$ and $A\in\operatorname{MAlg}(G^{(0)})$, define $\alpha\cdot A=r(s|_\alpha^{-1}(A))$. This defines an (isometric, order-preserving) action of $[G]$ on $\operatorname{MAlg}(G^{(0)})$. This action also extends to ultraproducts of full groups and measure algebras.

\begin{theorem}\label{theoremfirstimportant}
An aperiodic pmp groupoid $G$ is sofic if and only if the full group $[G]$ is metrically sofic. More precisely, every isometric embedding of $[G]$ into an ultraproduct $\prod_\mathcal{U}\mathfrak{S}_n$ extends uniquely to an isometric embedding of $[[G]]$.
\end{theorem}
\begin{proof}
Let's deal with uniqueness first: If $\theta:[[G]]\to\prod_\mathcal{U}\mathfrak{S}_n$ is an isometric embedding, then for every $\alpha\in [[G]]$ choose, by Lemmas \ref{lemmaaperiodicsupport} and \ref{lemmaextensionfullsemigrouptofullgroup}, $\widetilde{\alpha},\beta\in [G]$ with $\operatorname{supp}\beta=s(\alpha)$ and $\alpha\subseteq\widetilde{\alpha}$. Then $\theta(\alpha)=\theta(\widetilde{\alpha})\operatorname{supp}\theta(\beta)$, so $\theta$ is uniquely determined by its restriction to $[G]$.

Now suppose $\theta:[G]\to\prod_\mathcal{U}\mathfrak{S}_n$ is an isometric embedding, and let's use the ideas above to extend it to $[[G]]$.

Given $A\in\operatorname{MAlg}(G^{(0)})$, choose $\alpha\in[G]$ with $\operatorname{supp}(\alpha)=A$ and define $\phi(A)=\operatorname{supp}(\theta(\alpha))$. We need several steps to finish this proof, namely,

\begin{enumerate}
\item\label{itemwelldefined} $\phi$ is well-defined, i.e., $\phi(A)$ does not depend on the choice of $\alpha$ with $\operatorname{supp}(\alpha)=A$:

Suppose $\alpha,\alpha'\in[G]$ satisfy $\operatorname{supp}\alpha=\operatorname{supp}\alpha'=A$. Consider any $\beta\in[G]$ with $\operatorname{supp}\beta=G^{(0)}\setminus A$. By Lemma \ref{lemmasupports}.\ref{lemmasupportandfix}, $\operatorname{supp}(\theta(\alpha))=\operatorname{fix}(\theta(\beta))=\operatorname{supp}(\theta(\alpha))$.
%
%
%
\item $\phi$ preserves intersections:

Let $A,B\in\operatorname{MAlg}(G^{(0)})$, and consider $\alpha,\beta,\gamma\in[R]$ with $\operatorname{supp}(\alpha)=A\cap B$, $\operatorname{supp}(\beta)=A\setminus B$ and $\operatorname{supp}(\gamma)=B\setminus A$. By Lemma \ref{lemmasupports}.\ref{lemmadisjointsupports}, the supports of $\theta(\alpha)$ and $\theta(\beta)$ are disjoint, so $\operatorname{supp}(\theta(\alpha\beta))=\operatorname{supp}(\theta(\alpha)\theta(\beta))=\operatorname{supp}(\theta(\alpha))\cup\operatorname{supp}(\theta(\beta))$, and similarly for $\alpha$ and $\gamma$. Also, $\operatorname{supp}(\alpha\beta)=A$ and $\operatorname{supp}(\alpha\gamma)=B$, so again by Lemma \ref{lemmasupports}.\ref{lemmadisjointsupports},
\begin{align*}
\phi(A)\cap\phi(B)&=\operatorname{supp}(\theta(\alpha\beta))\cap\operatorname{supp}(\theta(\alpha\gamma))\\
&=(\operatorname{supp}(\theta(\alpha))\cup\operatorname{supp}(\theta(\beta)))\cap(\operatorname{supp}(\theta(\alpha))\cup\operatorname{supp}(\theta(\gamma)))\\
&=\operatorname{supp}(\theta(\alpha))=\phi(A\cap B).
\end{align*}
\item $\phi$ preserves measure:

By Proposition \ref{theoremtraceanddistance}, $\theta$ is trace-preserving, so $\phi$ preserves measure.
\item If $\alpha\in[G]$, then $\phi(\operatorname{fix}\alpha)=\operatorname{fix}(\theta(\alpha))$:

We need just to verify that $\phi$ preserves complements. Given $A\in\operatorname{MAlg}(G^{(0)})$, the complement $B=G^{(0)}\setminus A$ is the unique element disjoint with $A$ and such that $\mu(A)+\mu(B)=1$, and all of this is preserved by $\phi$.

\item $(\theta,\phi)$ is covariant:

Let $A\in\operatorname{MAlg}(G^{(0)})$ and $\alpha\in[G]$. Take $\beta\in[G]$ with $\operatorname{supp}\beta=A$. Then $\operatorname{supp}(\alpha\beta\alpha^{-1})=\alpha\cdot A$, and
\begin{align*}
\phi(\alpha\cdot A)&=\operatorname{supp}(\theta(\alpha\beta\alpha^{-1}))=\operatorname{supp}(\theta(\alpha)\theta(\beta)\theta(\alpha)^{-1})=\theta(\alpha)\cdot\operatorname{supp}(\theta(\beta))\\
&=\theta(\alpha)\cdot\phi(A).
\end{align*}
\end{enumerate}

For every $\alpha\in[[G]]$, set $\Phi(\alpha)=\theta(\alpha')\phi(s(\alpha))$, where $\alpha'\in[G]$ is such that $\alpha\subseteq\alpha'$ (which exists by Lemma \ref{lemmaextensionfullsemigrouptofullgroup}). Using the definition of $\phi$ and the fact that it preserves the order, it is not hard to see that $\Phi$ is also well-defined. Note that $\Phi$ extends both $\theta$ and $\phi$.

Let's show that $\Phi$ is a sofic embedding. Suppose $\alpha=\alpha'A$, $\beta=\beta'B$, where $\alpha,\beta\in[[G]]$, $\alpha',\beta'\in[G]$ and $A,B\in\operatorname{MAlg}(G^{(0)})$. Then
\[\alpha\beta=\alpha'\beta'(B\cap\beta^{-1}\cdot A)\]
Since the same kind of formula holds on ultraproducts and $(\theta,\phi)$ is a covariant pair of morphisms, we obtain $\Phi(\alpha\beta)=\Phi(\alpha)\Phi(\beta)$.

It remains only to see that $\Phi$ is trace-preserving. Let $\alpha\in[[G]]$. If we show that $\operatorname{fix}\Phi(\alpha)=\phi(\operatorname{fix}\alpha)$ we are done because $\phi$ is isometric.

Let $\alpha'\in[G]$ with $\alpha\subseteq\alpha'$. Let $A=\operatorname{fix}\alpha$ and $B=\operatorname{fix}\alpha'\setminus A$. Note that $A=\operatorname{fix}\alpha'\cap s(\alpha)$, so \[\phi(A)=\operatorname{fix}\theta(\alpha')\cap s(\Phi(\alpha)).\]
Since $\Phi(\alpha)=\theta(\alpha')\phi(s(\alpha))\subseteq\theta(\alpha')$, we have $\operatorname{fix}(\Phi(\alpha))=\operatorname{fix}(\theta(\alpha'))\cap s(\Phi(\alpha))$. Thus we are done.
\qedhere
\end{proof}

Now we extend this result to when $G$ has periodic points, but no singleton classes. Set $\operatorname{Per}_{\geq 2}(G)=\left\{x\in G^{(0)}:2\leq|s^{-1}(x)|<\infty\right\}$.

\begin{lemma}
There exists $\alpha\in [G]$ with $\operatorname{supp}\alpha=\operatorname{Per}_{\geq 2}(G)$.
\end{lemma}
\begin{proof}
This follows easily from the existence of a transversal for periodic relations (\cite[Theorem 12.16]{kechrisclassicaldescriptivesettheory}) and an argument similar to the proof of Lemma \ref{lemmaaperiodicsupport}.\qedhere
\end{proof}

\begin{theorem}
Suppose that for all $x\in G^{(0)}$, $|s^{-1}(x)|\geq 2$. Then $G$ is sofic if and only if $[G]$ is metrically sofic.
\end{theorem}
\begin{proof}
Let $P=\operatorname{Per}_{\geq 2}(G)$ and $\operatorname{Aper}=G^{(0)}\setminus P$, and consider the subgroupoid $H=G\operatorname{Aper}$ of $G$. By previous results, it suffices to show that $[H]$ is metrically sofic. Fix any $\rho\in[G]$ with $\operatorname{supp}\rho=P$. Let $\theta:[G]\to\prod_{\mathcal{U}}\mathfrak{S}_{n_k}$ be an isometric embedding.

Consider the embedding $[H]\to[G]$, $\alpha\mapsto \widetilde{\alpha}=\alpha\cup P$. This embedding modifies distances by a multiplicative factor of $\mu(\operatorname{Aper})$. By Lemma \ref{lemmasupports}.\ref{lemmadisjointsupports}, $\operatorname{supp}\theta(\widetilde{\alpha})\subseteq\operatorname{fix}\theta(\rho)$. We can then restrict $\theta(\alpha)$ to $\operatorname{fix}\theta(\rho)$ (similarly to how we did in Theorem \ref{theorempermanence}.\ref{theorempermanenceitemsubgroupoid}.) and obtain a new embedding $\eta:[H]\to\prod_{\mathcal{U}}\mathfrak{S}_{m_k}$ (where $m_k\leq n_k$). This new embedding modifies distances by a multiplicative factor of
\[d(\theta(\rho),1)^{-1}=d(\rho,1)^{-1}=\mu(\operatorname{Aper})^{-1}\]
so $\eta$ is in fact isometric.\qedhere
\end{proof}

\bibliographystyle{abbrv}
\bibliography{biblio}

\end{document}